\documentclass[12pt]{article}
\usepackage{verbatim}
\usepackage{latexsym}
\usepackage{amsfonts}
\usepackage{amsmath}
\usepackage{amsthm}
\usepackage{amssymb}
\usepackage{mathtools}
\usepackage{xcolor}

\newtheorem{theorem}{Theorem}
\newtheorem{lemma}[theorem]{Lemma}

\newtheorem{open}{Open Problem}

\newcommand{\dist}{\mathrm{dist}}
\newcommand{\rev}{\mathrm{rev}}

\begin{document}

\title{Subsets of free groups with distinct differences}
\author{Simon R. Blackburn, Emma Smith\thanks{Research supported by UKRI and EPSRC as part of the Centre for Doctoral Training in Cyber Security at Royal Holloway, University of London (Grant Ref. EP/SO21817/1).} \ and Luke Stewart\footnotemark[1]\\
Department of Mathematics\\
Royal Holloway University of London\\
Egham, Surrey TW20 0EZ, United Kingdom}
\maketitle

\begin{abstract}
Let $F_n$ be a free group of rank $n$, with free generating set $X$. A subset $D$ of $F_n$ is a \emph{Distinct Difference Configuration} if the differences $g^{-1}h$ are distinct, where $g$ and $h$ range over all (ordered) pairs of distinct elements of $D$. The subset $D$ has diameter at most $d$ if these differences all have length at most $d$. When $n$ is fixed and $d$ is large, the paper shows that the largest distinct difference configuration in $F_n$ of diameter at most $d$ has size approximately $(2n-1)^{d/3}$. 
\end{abstract}

\section{Introduction}
\label{sec:introduction}

Let $G$ be a group. A subset $D$ of $G$ is a \emph{Distinct Difference Configuration}, or \emph{DDC}, if the differences $g^{-1}h$ are distinct, where $g$ and $h$ range over all (ordered) pairs of distinct elements of $D$. For example, when $G$ is the free group $F_X$ on a set $X$, then $D=X\cup X^{-1}$ is a DDC (because the word $g^{-1}h$ for distinct $g,h\in D$ is always reduced, so $g$ and $h$ can be recovered from the reduced representative of the difference). However, the set $X\cup X^{-1}\cup\{e\}$, where $e$ is the identity element of $G$, is not a DDC since the difference of $x$ and $e$ and the difference of $e$ and $x^{-1}$ are equal for any element $x\in X$. As another example of a DDC, in a finite group: the subset $\{1,2,4\}$ is a DCC in the additive group of integers modulo $7$.

Sets with distinct differences have been studied for many decades, beginning with the study of abelian groups. For example, in 1942 Bose~\cite{Bose} constructed a set of $q$ elements in a finite field $\mathbb{F}_{q^2}$ whose (additive) differences are distinct; see Colbourn and Dinitz~\cite[Section~IV.14.2]{ColbournDinitz}, where the term \emph{disjoint difference set} is used. Sidon sets, namely sets of positive integers whose pairwise sums are distinct, have been studied since the 1930s~\cite{Sidon32}, originally motivated by a problem in Fourier series. (Note that Sidon sets and distinct difference configurations are equivalent notions for abelian groups.)  Erd\H os and Turan~\cite{ErdosTuran} and Singer~\cite{Singer} showed that the maximum number of elements in a Sidon set contained in the interval $[1,x]$ is at least $(1-\epsilon)\sqrt{x}$ and is at most $\sqrt{x}+O(x^{1/4})$.  Difference sets, namely subsets $D$ of a group $G$ such that every non-identity element of $G$ occurs an equal number of times as the difference of elements of~$D$, have been intensively studied in both abelian and non-abelian cases; see Colbourn and Dinitz~\cite[Chapters~IV.12 and~IV.13]{ColbournDinitz}. Many other variants have been studied (motivated by design theory or other applications), for example Costas arrays~\cite{WarnkeCorrell}, Golomb rulers~\cite{Drakakis}, difference families~\cite[Chapter~IV.10]{ColbournDinitz} and difference matrices~\cite[Chapter~IV.11]{ColbournDinitz}.

The results of Erd\H os--Turan and Singer can be thought of as providing bounds on the cardinality of a distinct difference configuration $D\subseteq\mathbb{Z}$ of bounded diameter (so the distance $|h-g|$ between any two elements $g,h\in D$ is small). The natural analogue of diameter for general groups can be defined in terms of a Cayley graph of the group as follows. Let $X$ be a generating set for~$G$. For elements $g,h\in G$, the \emph{distance} $\dist(g,h)$ between $g$ and $h$ is defined to be the smallest non-negative integer $\ell$ such that $h=gy_1y_2\cdots y_\ell $, where $y_1,y_2,\ldots, y_\ell\in X\cup X^{-1}$. (So $\dist(g,h)$ is the distance between nodes labelled $g$ and $h$ in the Cayley graph on $G$ with generating set $X\cup X^{-1}$. Alternatively, $\dist(g,h)$ is the distance of the difference $g^{-1}h$ from the identity.) For a subset $D\subseteq G$, the \emph{diameter} of $D$ is the minimum integer~$d$ such that $\dist(g,h)\leq d$ for all $g,h\in D$ (if this integer exists).

Recall that if $\alpha$ and $\beta$ are positive real functions of a positive integer $d$, we write $\alpha=\Theta(\beta)$ to mean that there exist positive constants $c_1$ and $c_2$, independent of $d$, such that $c_1 \beta\leq\alpha \leq c_2 \beta$ whenever $d$ is sufficiently large.

For a group $G$ with a finite generating set $X$ of cardinality $n$, it is natural to ask what is the maximum cardinality $m$ of a DDC with diameter $d$. The work of Erd\H os--Turan and Singer implies that $m=\Theta(\sqrt{d})$ when $G=\mathbb{Z}$ and $X=\{1\}$. Blackburn et al.~\cite[Theorem~9, Corollary~27]{BlackburnEtzion10} show that when $G=\mathbb{Z}^2$ and $X=\{(1,0),(0,1)\}$ then $m=\Theta(d)$. Stewart~\cite[Chapter~8]{Stewart} shows that when $G=\mathbb{Z}^n$ and $X$ is the standard generating set of size $n$ then $m=\Theta(d^{n/2})$. To understand these results, we note that a DDC of diameter $d$ and cardinality $m$ gives rise to $m(m-1)$ distinct differences, all lying at distance at most $d$ from the identity $e$. Thus, $m(m-1)\leq |B_d(e)|$, where $B_d(e)$ is the \emph{ball of radius $d$ about $e$} (namely the set of all elements of $G$ at distance at most $d$ from  $e$). Since $|B_d(e)|=\Theta(d^n)$ when $G=\mathbb{Z}^n$, these results say that $m=\Theta(\sqrt{|B_d(e)|})$ in the free abelian group of rank $1$, $2$ and $n$ respectively. This can be interpreted as saying that the bound $m(m-1)\leq |B_d(e)|$ is an important restriction on the size of a DDC of diameter $d$ in these groups.

In this paper, we aim to investigate distinct difference configurations of fixed diameter when the group $G$ is a free (non-abelian) group of finite rank. The motivation for considering the free group in this paper is four-fold: First, solving the problem for a free group of rank $n$ has implications for all $n$-generator groups (see the discussion in Section~\ref{sec:conclusion}). Second, the free group is a natural group to consider when moving to the non-abelian case (in analogy to the free abelian cases considered previously). Thirdly, we believe the results are surprising. Our final motivation comes from an application to key distribution in wireless sensor networks, where sensors are arranged in a regular tree. (The paper~\cite{BlackburnEtzion08} was motivated by a similar application, where sensors are arranged in a grid.)

This paper proves the following theorem. 
\begin{theorem}
\label{thm:free}
Let  $F_X$ be a free group, freely generated by a set $X$ of \linebreak cardinality $n$, where $n\geq 2$. Let $m(n,d)$ be the maximum cardinality of a DDC of diameter at most $d$ in  $F_X$. As $d\rightarrow\infty$ with $n$ fixed, then
\[
m(n,d)= (2n-1)^{d/3+O(\log d)}
\]
{\rm (}where the implicit constants might depend on $n${\rm )}.
\end{theorem}
In fact, we will show an upper bound of the form $m(n,d)\leq (2n-1)^{d/3+O(1)}$. We comment that our lower bound is probabilistic.

Roughly speaking the theorem says that, in the free group, $m(n,d)\approx (2n-1)^{d/3}$ when $n$ is fixed. This result is rather surprising: The ball $B_d(e)$ of radius $d$ about the identity in the free group has cardinality
\[
1+\sum_{i=1}^{d}2n(2n-1)^{i-1} =\Theta((2n-1)^d)
\]
and so, extrapolating from the results on the free abelian case above, you might guess that the right answer should be $m(n,d)\approx \sqrt{|B_d(e)|}\approx (2n-1)^{d/2}$. (You could prove an upper bound of this form by using the inequality $m(m-1)\leq |B_d(e)|$.) Theorem~\ref{thm:free} shows that in fact the correct order of magnitude of $m(n,d)$ is much smaller than this, so there are more important bounds on the cardinality of a DDC of diameter $d$ in the free group than the bound $m(m-1)\leq |B_d(e)|$.

How big is a subset of diameter $d$ in the free group, irrespective of it having distinct differences? For any element $h$ in a group $G$, it is easy to see that the ball $B_r(h)$ of radius $r$ and centre $h$ in $G$ is a subset of $G$ of diameter at most $2r$. When $G$ is the free group $F_X$ with free generating set $X$ and when $d$ is even, it is not hard to see that every subset $D$ of  $F_X$ of diameter $d$ is contained in a ball of radius $d/2$ (to see this, choose the centre $h$ of the ball to be the unique element of $F_X$ that is equidistant from a pair of elements of $D$ at distance $d$). So a subset of $F_X$ of diameter $d$ has cardinality at most
\[
|B_{d/2}(h)|=1+\sum_{i=1}^{d/2}2n(2n-1)^{i-1}=\Theta((2n-1)^{d/2})
\]
when $d$ is even; this bound is tight since $B_{d/2}(h)$ has diameter $d$. When $d$ is odd, we may similarly argue that a subset of $F_X$ of diameter $d$ has cardinality at most $|B_{(d+1)/2}(h)|=\Theta((2n-1)^{d/2})$. (This bound is no longer tight, as $B_{(d+1)/2}(h)$ does not have diameter $d$. A tight bound is not hard to prove; the analogue of the ball of radius $d/2$ in this situation is the union of two balls of radius $\lfloor d/2\rfloor$ whose centres are at distance $1$.) We conclude that the largest subsets of diameter $d$ in $F_X$ are of size approximately $(2n-1)^{d/2}$, which (by Theorem~\ref{thm:free}) is much larger than the size $m(n,d)$ of a DDC of diameter $d$ of maximal cardinality. 

We have now seen two elementary arguments showing that $m(n,d)$ grows no faster than $(2n-1)^{d/2}$ (approximately), namely using the bound $m(m-1)\leq |B_d(e)|$ or using the above bound on the cardinality of diameter $d$ subsets. There is also a lower bound on $m(n,d)$ (again weaker than the bound implied by Theorem~\ref{thm:free}) that is a corollary of the the following explicit construction of a DDC $D$. For simplicity, assume that $d$ is divisible by $4$. Let $R$ be the set of reduced words of length $d/4$. For a reduced word $w$, let $\rev(w)$ be the reverse of $w$ (so $\rev(w)$ is a reduced word, with symbols listed in the opposite order to $w$). For any $w\in R$, the concatenation $w\,\rev(w)$ is a reduced word of length $d/2$. Define $D=\{w\,\rev(w):w\in R\}$. Now $D\subseteq B_{d/2}(e)$, where $e$ is the identity element, and so $D$ has diameter at most $d$. Furthermore, the following argument shows that all differences in $D$ are distinct. Suppose that $y,y'\in D$ with $y\neq y'$. Then $y=w\,\rev(w)$ and $y'=w'\,\rev(w')$ for distinct $w,w'\in R$. The difference $(y')^{-1}y$ is of the form $\rev(y')^{-1}(y')^{-1}y\,\rev(y)$. Since $y,y'$ are distinct and reduced, the expression $(y')^{-1}y$ does not reduce to the identity, and so the reduced form of $(y')^{-1}y$ begins with $d/4$ symbols representing $\rev(y')^{-1}$ and ends with $d/4$ symbols representing $\rev(y)$. So $\rev(y)$ and $\rev(y')$, and hence $y$ and $y'$, are determined by the reduced form of $(y')^{-1}y$ . Thus the differences are all distinct, and $D$ is a DDC. Since $|D|=|R|=2n(2n-1)^{d/4-1}$, we have constructed a DDC of diameter $d$ and cardinality approximately equal to the cardinality of the ball $B_{d}(e)$ raised to the power $1/4$.

The remainder of the paper is structured as follows. In Section~\ref{sec:upper_bound} we prove the upper bound of Theorem~\ref{thm:free}. In Section~\ref{sec:lower_bound} we provide a probabilistic argument to establish the lower bound in Theorem~\ref{thm:free} (finishing the proof of the theorem).
Finally, in Section~\ref{sec:conclusion} we provide some commentary, and list some open problems.

\section{An upper bound}
\label{sec:upper_bound}
In this section we provide an upper bound on the maximum cardinality $m(n,d)$ of a DDC with diameter $d$ contained in the free group of rank~$n$. This gives the upper bound in Theorem \ref{thm:free}. We require the following definitions. Let  $F_X$ be a free group, finitely generated by a set $X$ of \linebreak cardinality $n$. Let $D$ be a DDC in  $F_X$ with diameter $d$. Recall that for even $d$, the set $D$ is contained in a ball of radius $d/2$ with centre an element $g \in  F_X$, denoted $B_{d/2}(g)$. Indeed, we may assume (for even $d$) that $D\subseteq B_{d/2}(e)$ because we may replace $D$ by $g^{-1}D$ without altering the distinct difference property or diameter. Similarly, for odd $d$, we may assume that $D\subseteq B_{(d+1)/2}(e)$.

Define the \textit{sphere of radius $r$ about e}, denoted $S_{r}(e)$, by $S_{r}(e) = B_{r}(e)\setminus B_{r-1}(e)$. Equivalently, $S_{r}(e)$ denotes all elements at distance exactly $r$ from $e$, the `outer' elements of the ball. For every $x \in B_{r}(e)$, define $D_{x} = \{x'\colon xx' \in D, xx' \text{ is reduced}\}$.  So $D_{x}$ consists of the strings following $x$ in the reduced words representing elements of $D$.

\begin{lemma}
\label{lemma: condition} 
Let $D \subseteq S_{r}(e)$ where $r$ is a positive integer.  If $D$ is a DDC, then
\[
\lvert D_{x} \cap D_{y}\rvert \leq 1 \text{ for all distinct }x,y \in B_{r-1}(e).
\]
\end{lemma}

\begin{proof}
Assume that for some $x,y \in B_{r-1}(e)$, where $x \neq y$, we have $\lvert D_{x} \cap D_{y}\rvert \geq 2$. Then there must exist $z,w \in D_{x}$ and $z,w \in D_{y}$, where $z \neq w$. Therefore, there must be elements $a=xz, b=xw, u=yz, v=yw$ that are reduced, distinct elements of $D$. Then, $a^{-1}b=z^{-1}x^{-1}xw=z^{-1}y^{-1}yw=u^{-1}v$, and so $D$ does not have pairwise distinct differences. Therefore, if $D$ is a DDC, then $\lvert D_{x} \cap D_{y}\rvert \leq 1$ for all $x,y \in B_{r-1}(g)$ where $x \neq y$.
\end{proof}

Lemma \ref{lemma: condition} provides a necessary condition for a set $D$ contained in a sphere to be a DDC. This condition is in fact both necessary and sufficient: see \cite[Theorem 6.0.5]{Stewart}. We now prove the main theorem of this section.

\begin{theorem} \label{thm: upper bound}
Let  $F_X$ be a free group, freely generated by a set $X$ of cardinality $n$, where $n\geq 2$. Let $m(n,d)$ be the maximum cardinality of a DDC of diameter $d$ in  $F_X$. As $d\rightarrow\infty$ with $n$ fixed, there exists a positive constant $c$  {\rm (}independent of $d$, but depending on $n${\rm )} such that
\[
m(n,d)\leq c (2n-1)^{d/3}
\]
for all sufficiently large $d$.
\end{theorem}
 In fact, we show that
\[
|D| \leq \frac{2n(4n^{2}-3n+1) }{(2n-1)^{1/3}((2n-1)^{2/3}-1)} \cdot (2n-1)^{d/3}.
\]
\begin{proof}
Let $D \subseteq F_X$ be a DDC of diameter $d$. Without loss of generality, replacing $D$ by $g^{-1}D$ if necessary for some $g \in  F_X$, we may assume that $D \subseteq B_{\lceil d/2 \rceil}(e)$. Let $j$ be a non-negative integer and let $D^{(j)}$ denote the set $D \cap S_{j}(e)$. Then $D = \bigcup_{j=0}^{\lceil d/2\rceil} D^{(j)}$. Note that $|D^{(0)}| \leq 1$, and therefore
\begin{equation} \label{eqn: upper bound 1}
|D| \leq 1 + \sum_{j=1}^{\lceil d/2 \rceil} |D^{(j)}|.
\end{equation}
We now prove an upper bound on the cardinality of the sets $D^{(j)}$. If $k_{j}$ is a non-negative integer strictly less than $j$, then it is not difficult to see that $D^{(j)}= \bigcup_{x \in S_{k_{j}}(e)}x\cdot D_{x}^{(j)}$, where $D_{x}^{(j)}$ is defined as above Lemma \ref{lemma: condition}. We choose $k_{j}$ to be $\lfloor j/3 \rfloor$. The sets $x\cdot D_{x}^{(j)}$ are distinct for each $x \in S_{\lfloor j/3 \rfloor}(e)$, so we have
\begin{equation} \label{eqn: upper bound 2}
|D^{(j)}| = \sum_{x \in S_{\lfloor j/3 \rfloor}(e)} |D_{x}^{(j)}|. \end{equation}
By the inclusion-exclusion principle, we have
\begin{equation} \label{eqn: upper bound 3}
\left| \bigcup_{x \in S_{\lfloor j/3 \rfloor}(e)} D_{x}^{(j)} \right| \geq \sum_{x \in S_{\lfloor j/3 \rfloor}(e)} \left| D_{x}^{(j)}\right|  - \!\!\! \sum_{\substack{x, y \in S_{\lfloor j/3 \rfloor}(e), \\ x \neq y}} \left| D_{x}^{(j)} \cap D_{y}^{(j)}\right| .
\end{equation}
Combining inequality (\ref{eqn: upper bound 1}), equation (\ref{eqn: upper bound 2}) and inequality (\ref{eqn: upper bound 3}), we have
\begin{equation} \label{eqn: upper bound 4}
\left| D \right| \leq 1 + \sum_{j=1}^{\lceil d/2 \rceil} \left(  \left| \bigcup_{x \in S_{\lfloor j/3 \rfloor}(e)} D_{x}^{(j)} \right| + \!\!\! \sum_{\substack{x, y \in S_{\lfloor j/3 \rfloor}(e), \\ x \neq y}} \left| D_{x}^{(j)} \cap D_{y}^{(j)}\right| \right).
\end{equation}

For all $x \in S_{\lfloor j/3 \rfloor}(e)$, every element in $D_{x}^{(j)}$ has length $j-\lfloor j/3 \rfloor$. Therefore $ \bigcup_{x \in S_{\lfloor j/3 \rfloor}(e)} D_{x}^{(j)} \subseteq  S_{j-\lfloor j/3 \rfloor}(e)$. As $j-\lfloor j/3 \rfloor \leq 2j/3 + 1$, we have $\lvert \bigcup_{x \in S_{\lfloor j/3 \rfloor}(e)} D_{x}^{(j)} \rvert \leq  2n(2n-1)^{2j/3}$.

Each set $D^{(j)}$ is a DDC where every element has equal length $j$. \linebreak Lemma \ref{lemma: condition} tells us that for every $j$, we have $|D_{x}^{(j)} \cap D_{y}^{(j)}| \leq 1$ for all distinct $x,y \in S_{\lfloor j/3 \rfloor}(e)$. Using the fact that
\[
{| S_{k_{j}}(e) | \choose 2} < \tfrac{1}{2} \left(2n(2n-1)^{2\lfloor j/3 \rfloor-2}\right)^{2},
\]
and $2\lfloor j/3 \rfloor \leq 2j/3$, inequality (\ref{eqn: upper bound 4}) now becomes
\begin{equation} \label{eqn: upper bound 5}
|D| \leq 1 + \sum_{j=1}^{\lceil d/2 \rceil} \left( 2n(2n-1)^{2j/3} + 2n^{2}(2n-1)^{2j/3-2} \right).
\end{equation}

We can rearrange inequality \eqref{eqn: upper bound 5} to get

\begin{equation} \label{eqn: upper bound 6}
|D| \leq 1 + \frac{2n(4n^{2}-3n+1)}{(2n-1)^{2}} \sum_{j=1}^{\lceil d/2 \rceil} (2n-1)^{2j/3}. 
\end{equation}

The summation in (\ref{eqn: upper bound 6}) forms a geometric sequence, and as $\lceil d/2 \rceil \leq d/2 + 1$ we find
\begin{equation} 
|D| \leq  1 + \frac{2n(4n^{2}-3n+1)}{(2n-1)^{2}} \cdot \frac{(2n-1)^{2/3}((2n-1)^{d/3 + 1}-1)}{(2n-1)^{2/3}-1}. \label{eqn: upper bound 7}
\end{equation}

We rewrite inequality (\ref{eqn: upper bound 7}) to see

\begin{equation*} 
|D| \leq 1 + \frac{2n(4n^{2}-3n+1)}{(2n-1)^{4/3}((2n-1)^{2/3}-1)} + \frac{2n(4n^{2}-3n+1)(2n-1)^{d/3} }{(2n-1)^{1/3}((2n-1)^{2/3}-1)} 
\end{equation*}
Since 
\[
\frac{2n(4n^{2}-3n+1)}{(2n-1)^{4/3}((2n-1)^{2/3}-1)} \geq 1,
\]
we get the desired bound of
\begin{equation*} 
|D| \leq \frac{2n(4n^{2}-3n+1) }{(2n-1)^{1/3}((2n-1)^{2/3}-1)} \cdot (2n-1)^{d/3}. \qedhere
\end{equation*}
\end{proof}

Theorem \ref{thm: upper bound} proves the upper bound in Theorem \ref{thm:free}.

\section{A lower bound}
\label{sec:lower_bound}
In this section we prove a probabilistic result which forms the lower bound on $m(n,d)$, the maximum cardinality of a DDC of diameter $d$ in the free group of rank $n$.

Let  $F_X$ be a free group, freely generated by a set $X$ of cardinality $n$, where $n\geq 2$. The sphere $S_r(e)$ of radius $r$ centred at the identity $e$ is a subset of  $F_X$ with diameter $2r$. We show the existence of a large DDC in this sphere. Firstly, we prove that any three-element subset of the free group whose elements have equal length must have pairwise distinct differences.

\begin{lemma}
\label{lem:3-sets}
Let $F_X$ be a free group, freely generated by a set $X$ of cardinality~$n$, where $n\geq 2$. Let $r$ be a positive integer. If $D$ is a subset of $S_r(e)$ with cardinality $3$ then $D$ is a DDC.
\end{lemma}
\begin{proof}
Write $D= \{a,b,c\}$ where $a,b,c\in S_r(e)$ are distinct elements of  $F_X$ of length~$r$. Assume, for a contradiction, that $D$ is not a DDC. Then two differences in $D$ agree. If these differences involve just two elements, $a$ and $b$ say, then we must have $a^{-1}b=b^{-1}a$. But this implies that $(ab^{-1})^2=e$. Since  $F_X$ has no elements of order $2$, we see that $ab^{-1}=e$ and so $a=b$. This contradicts the fact that $a$ and $b$ are distinct. So all three elements must be involved in our two agreeing differences in $D$. Up to relabelling, either $a^{-1}b=a^{-1}c$, or $a^{-1}b=c^{-1}a$. If $a^{-1}b=a^{-1}c$, then by left-multiplying by $a$ we see $b=c$ which is again a contradiction. So we may assume that $a^{-1}b=c^{-1}a$ and $a,b,c$ are distinct.

If $a^{-1}b=c^{-1}a$, then the reduced words $a^{-1}b$ and $c^{-1}a$ must have the same length. If there is no reduction, then by comparing letters we see $a=c$ which is a contradiction. So there must exist a positive integer $k \leq r$ such that the reduced words $a^{-1}b$ and $c^{-1}a$ both have length $2r-2k$. Write $a_i$, $b_i$ and $c_i$ for the $i$th letter in the words $a$, $b$ and $c$ respectively. Then, we have $a_{1}\cdots a_{k}=b_{1}\cdots b_{k}$. By comparing letters on the right hand side of the reduced words corresponding to $a^{-1}b$ and $c^{-1}a$, we also get $b_{k+1}\cdots b_{r}=a_{k+1}\cdots a_{r}$. But then $a=b$ which is a contradiction.
\end{proof}

Lemma \ref{lem:3-sets} tells us that when considering subsets of the free group where all elements have the same length, if the subset is not a distinct difference configuration then the repeated difference must arise from the pairwise differences of four distinct elements. We now prove the lower bound.

\begin{theorem}
\label{thm: lower bound}
Let  $F_X$ be a free group, freely generated by a set $X$ of cardinality~$n$, where $n\geq 2$. Let $m(n,d)$ be the maximum cardinality of a DDC of diameter at most $d$ in  $F_X$. As $d\rightarrow\infty$ with $n$ fixed, 
\[
m(n,d) \geq (2n-1)^{d/3+ O(\log (d))}.
\]
\end{theorem}
\begin{proof}
We will show that when $d$ is even, we have
\begin{equation} \label{eqn: lower bound 1}
m(n,d) \geq 2n(2n-1)^{\frac{d}{3} -\frac{4}{3} \log_{2n-1}(d/3)-5}.
\end{equation}
The left hand side of inequality (\ref{eqn: lower bound 1}) implies a lower bound of the form $(2n-1)^{d/3+ O(\log (d))}$. Therefore once we have established inequality (\ref{eqn: lower bound 1}) for $d$ even, the theorem is proved for when $d$ is even. If $d$ is odd, then $m(n,d) \geq m(n,d-1)$ where $d-1$ is even. Hence (\ref{eqn: lower bound 1}) shows the theorem also holds when $d$ is odd. It therefore suffices to prove inequality (\ref{eqn: lower bound 1}) where we may assume that $d$ is even.

Let $\gamma$ be such that $\frac{d}{3}-\gamma$ is a positive integer. We will choose the specific value of $\gamma$ later (which will be logarithmic in $d$). We construct a randomised set as follows.

Let $V=S_{d/3 - \gamma}(e)$, and let $W=S_{d/6 + \gamma}(e)$. For all $v \in V$, choose $w_{v} \in W$, uniformly and independently at random so that $vw_{v}$ is reduced. Note that there are $2n(2n-1)^{d/6 + \gamma}$ options for each $w_{v}$. Let $D=\{vw_{v} \mid v \in V \}$ be the resulting random subset of words in $F_{X}$.

Note that $|D|=|V|=2n(2n-1)^{d/3 - \gamma-1}$, because $vw_v\not = v'w_{v'}$ when $v\not=v'$. Also, $D \subseteq S_{d/2}(e)$ because $vw_{v}$ is always reduced. The diameter of $D$ is therefore less than or equal to $d$.

If $D$ is a set constructed as previously described, we wish to know how many duplicated differences we would expect. All elements in $D$ have equal length $\frac{d}{2}$ and so Lemma \ref{lem:3-sets} tells us that if there is a repeated difference, then \linebreak there must be four distinct elements $u,v,x,y$ in $V$ such that \linebreak $(uw_{u})^{-1}(vw_{v})=(xw_{x})^{-1}(yw_{y})$.

Let $A_{u,v,x,y}$ be the event that $(uw_{u})^{-1}(vw_{v})=(xw_{x})^{-1}(yw_{y})$, and let $I_{u,v,x,y}$ be the indicator random variable for event $A_{u,v,x,y}$. Let $\mathbb{P}_{u,v,x,y}$ be the probability that $I_{u,v,x,y}=1$.

Define $\eta (D) = \sum_{(u,v,x,y)}\mathbb{P}_{u,v,x,y}(D)$, where the sum is over ordered quadruples $(u,v,x,y)$ of distinct elements of $V$. By linearity of expectation, $\eta (D)$ is the expected number of bad events that occur. We aim to show that \begin{equation} \label{eqn: lower bound 2}
\eta (D) \leq (d/3 -1 )2n(2n-1)^{d/3 - 4\gamma -1}.
\end{equation}
This suffices to prove the theorem, by the following argument.
There exists an outcome $D^{*} \subseteq S_{d/2}(e)$ where the number of repeated differences is at most $\eta (D)$. By removing at most $\eta (D)$ elements from $D^{*}$ (one element in each bad event), we produce a subset $D'$ which is a DDC of diameter at most $d$, and has cardinality at least $| D^{*}| - \eta (D)$. Hence
$m(n,d)\geq |D^*|-\eta (D)$. The inequality~\eqref{eqn: lower bound 2} therefore implies
\[
m(n,d) \geq 2n(2n-1)^{d/3 - \gamma -1} - (d/3 -1 )2n(2n-1)^{d/3 - 4\gamma -1}.
\]
We choose $\gamma$ to be an integer such that
\[
\frac{1}{3}\log_{2n-1}(d/3) \leq \gamma \leq \frac{1}{3}\log_{2n-1}(d/3) + 1.
\]
Using the lower bound on $\gamma$, we have $d/3 \leq (2n-1)^{3 \gamma}$ and therefore $m(n,d) \geq 2n(2n-1)^{d/3 -4 \gamma-1}$. Then using the upper bound on $\gamma$ we see that 
\[
m(n,d) \geq 2n(2n-1)^{\frac{d}{3} -\frac{4}{3} \log_{2n-1}(d/3)-5},
\]
proving inequality (\ref{eqn: lower bound 1}). So the theorem follows, once we have established~\eqref{eqn: lower bound 2}. We now show that this inequality holds.

Assume that for some distinct $u,v,x,y \in V$, we have
\[
(uw_{u})^{-1}(vw_{v})=(xw_{x})^{-1}(yw_{y}).
\]
As $u,v,x,y$ are distinct, there must be some cancellation, which must be equal on both sides. Therefore for some (strictly) positive integer $k$, we have $|(uw_{u})^{-1}(vw_{v})|=|(xw_{x})^{-1}(yw_{y})|= d-2k$. As $u \neq v$ and $x \neq y$, we have  $k < \frac{d}{3}-\gamma$. In other words, for some $0 < k< \frac{d}{3}-\gamma$ the following conditions hold:
\begin{align}
u_{1}\ldots u_{k}&=v_{1}\ldots v_{k}, \label{condition: 1}\\
x_{1}\ldots x_{k}&=y_{1}\ldots y_{k}, \label{condition: 2} \\
u_{k+1}\ldots u_{d/3 - \gamma}&=x_{k+1}\ldots x_{d/3 - \gamma}, \label{condition: 3} \\
v_{k+1}\ldots v_{d/3 - \gamma}&=y_{k+1}\ldots y_{d/3 - \gamma}, \label{condition: 4} \\
u_{k+1}&\neq v_{k+1}, \label{condition: 5} \\
x_{k+1}&\neq y_{k+1}. \label{condition: 6}
\end{align}
Note that, for a given choice of $u$, $v$, $x$ and $y$, conditions (\ref{condition: 1}) to (\ref{condition: 6}) are satisfied for at most one value of $k$. If $u$, $v$, $x$ and $y$ are such that  these conditions are never satisfied, whatever the value of $k$, then $\mathbb{P}_{u,v,x,y}(D)=0$. Otherwise, $\mathbb{P}_{u,v,x,y}(D)$ is equal to the probability that $w_{u}=w_{x}$ and $w_{v}=w_{y}$ and so $\mathbb{P}_{u,v,x,y}(D)\leq (2n-1)^{-2(d/6+\gamma)}$. Hence
\begin{equation}
\label{eqn:eta}
\eta(D)\leq \sum_{k}\sum_{u,v,x,y} (2n-1)^{-2(d/6+\gamma)},
\end{equation}
where the outer sum runs over all $k$ such that $0 < k< \frac{d}{3}-\gamma$, and the inner sum runs over all distinct $u,v,x,y\in V$ that satisfy conditions (\ref{condition: 1}) to (\ref{condition: 6}). There are at most $d/3-1$ choices for the value of $k$ in~\eqref{eqn:eta}. Fix a value of~$k$. There
are $2n(2n-1)^{d/3-\gamma-1}$ choices for $u$, and then $(2n-1)^{d/3-\gamma}$ choices for $y$ (as $y_{k+1}\neq u_{k+1}$ by~\eqref{condition: 3} and~\eqref{condition: 6}).
The choices of $u$ and $y$, the value of $k$,
 and the conditions~\eqref{condition: 1} to~\eqref{condition: 4} determine $v$ and $x$. So there are at most $2n(2n-1)^{2d/3 - 2\gamma -1}$ terms in the inner sum of~\eqref{eqn:eta}. 
Hence~\eqref{eqn:eta} implies
\[
\eta (D) \leq (d/3 -1 )2n(2n-1)^{2d/3 - 2\gamma -1}(2n-1)^{-2(d/6+\gamma)},
\] which proves the inequality (\ref{eqn: lower bound 2}) as required.
\end{proof} 

Theorem \ref{thm: lower bound} proves the lower bound of Theorem \ref{thm:free}. Combined with Theorem \ref{thm: upper bound}, which forms the upper bound, this completes the proof of Theorem~\ref{thm:free}.

\section{Discussion}
\label{sec:conclusion}

\begin{enumerate}
\item It would be interesting to tighten our bounds on $m(n,d)$, the largest DDC in $F_n$ of diameter at most $d$, and to remove the probabilistic nature of our lower bound:

\begin{open}
Is there a good explicit construction for a large DDC of diameter $d$ in $F_n$?
\end{open}

\begin{open}
As $d\rightarrow\infty$ with $n$ fixed, is it the case that $m(n,d)=(2n-1)^{d/3+O(1)}$?
\end{open}

\item The problem of determining $m(n,d)$ when $d$ is fixed is also interesting. Stewart~\cite{Stewart} shows that $m(n,2)=2n$, that $m(n,3)=2n+1$, and $m(n,4)=2\sqrt{2}n^{3/2}+o(n^{1.3})$. (The exponent of the error term comes from known results on the gaps between primes.)

\item Our upper bound on $m(n,d)$ has implications for all finitely generated groups. Indeed, we have the following theorem:

\begin{theorem}[Stewart~\cite{Stewart}]
\label{thm:quotient}
Let $G$ be a finitely generated group with finite generating set  $\{g_1,g_2,\ldots ,g_n\}$ of cardinality $n$. Let $D$ be a DDC in $G$ of diameter at most $d$. Then $|D|\leq m(n,2d)$, where $m(n,2d)$ is the maximum size of a DDC of diameter at most $2d$ in the free group $F_n$ of rank $n$.
\end{theorem}
\begin{proof}
Let $\{x_1,x_2,\ldots,x_n\}$ be a free generating set for $F_n$. Let $\phi:F_n\rightarrow G$ be the (unique) homomorphism such that $\phi(x_i)=g_i$ for $i\in\{1,2,\ldots,n\}$. Let $D=\{h_1,h_2,\ldots ,h_m\}$ be a DDC in $G$ of cardinality $m$ and of diameter at most $d$. Define $\hat{D}=\{\hat{h}_1,\hat{h}_2,\ldots ,\hat{h}_m\}\subseteq F_n$ as follows. Set $\hat{h}_1\in F_n$ to be any element such that $\phi(\hat{h}_1)=h_1$. For $i>1$, define $\hat{h}_i\in F_n$ to be any element such that $\phi(\hat{h}_i)=h_i$ and such that $\dist(\hat{h}_1,\hat{h}_i)\leq d$. We see that $\hat{D}$ is a DDC in $F_n$ of cardinality $|D|$ and diameter at most $2d$. Hence $|D|\leq m(n,2d)$ and the theorem follows.
\end{proof}

\begin{open}
Let $G$ be a finitely generated group with a finite generating set  $\{g_1,g_2,\ldots ,g_n\}$ of cardinality $n$. Let $D$ be a DDC in $G$ of diameter at most $d$. Is it the case that $|D|\leq m(n,d)$?
\end{open}
We note that this problem cannot be solved by naively pulling a DDC in $G$ back into the free group: in the notation of the theorem above, there are examples of subsets $D\subseteq G$ of diameter $d$ such that any preimage $\hat{D}\subseteq F_n$ with $\phi(\hat D)=D$ always has diameter greater than~$d$.

\item
Let $G$ be a group generated by a finite set $X$. We have defined the difference between two elements $g$ and $h$ in a group $G$ as $g^{-1}h$. We could define this as the \emph{right difference} between $g$ and $h$, and define the equally natural \emph{left difference} between $g$ and $h$ as $gh^{-1}$. So a \emph{left distinct difference configuration} is a subset $D$ whose left differences are distinct. (In fact, this is often the definition of a distince difference configuration used in the literature.) We define the \emph{left distance} between $g$ and $h$ as the length of the shortest expression of $gh^{-1}$ as a product of elements in $X\cup X^{-1}$, and the \emph{left diameter} of $D$ as the largest left distance between a pair of elements in $D$.  It is not hard to prove that that $D$ is a (right) distinct difference configuration of diameter $d$ if and only if $D^{-1}:=\{g^{-1}:g\in D\}$ is a left distinct difference configuration of left diameter $d$, so all the theory applies in this case. (Even though a right distinct difference configuration of right diameter $d$ is always a left distinct difference configuration, it might not have left diameter $d$. For example, the subset $D=\{x_1x_2,x_1x_3\}\subset F_3$ has right diameter $2$ but left diameter $4$.)
\end{enumerate}

\end{document}